\input ./style/arxiv-vmsta.cfg
\documentclass[numbers,compress,v1.0.1]{vmsta}
\usepackage{vtexbibtags}
\usepackage{authorquery}
\usepackage{mathrsfs}

\volume{5}% Updated by VTEXPTS2LaTeX.exe, 18.12.2018 09:31
\issue{4}% Updated by VTEXPTS2LaTeX.exe, 18.12.2018 09:31
\pubyear{2018}
\firstpage{483}% Updated by VTEXPTS2LaTeX.exe, 18.12.2018 09:31
\lastpage{499}% Updated by VTEXPTS2LaTeX.exe, 18.12.2018 09:31
\aid{VMSTA119}
\doi{10.15559/18-VMSTA119}% Updated by VTEXPTS2LaTeX.exe, 05.10.2018 09:31
\articletype{research-article}

\startlocaldefs
\newcommand{\rrVert}{\Vert}
\newcommand{\llVert}{\Vert}

\allowdisplaybreaks

\newtheorem{theorem}{Theorem}
\newtheorem{remark}{Remark}
\newtheorem{lemma}{Lemma}

\newtheorem{proposition}{Proposition}
\theoremstyle{definition}
\newtheorem{definition}{Definition}

\def\cl#1{{\mathscr{#1}}}
\def\P{{\mathbb{P}}}
\def\R{{\mathbb{R}}}
\def\N{{\mathbb{N}}}
\def\E{{\mathbb{E}}}
\def\Var{\mathrm{Var}}
\def\Cov{\mathrm{Cov}}
\def\ep{\varepsilon}
\endlocaldefs

\begin{aqf}
\querytext{Q1}{Note: The symbol \P prints out in an unfavorable way.}
\querytext{Q2}{Please check the re-edited sentence.}
\querytext{Q3}{Please check the re-edited sentence.}
\querytext{Q4}{Please check -ized/-ised in the original title.}
\querytext{Q5}{\label\string{eq:frac-eq-u\string}}
\end{aqf}
\begin{document}

\begin{frontmatter}
\pretitle{Research Article}

\title{Large deviations for conditionally Gaussian processes:
 estimates of level crossing probability}

\author[a]{\inits{B.}\fnms{Barbara}~\snm{Pacchiarotti}\thanksref{cor1}\ead[label=e1]{pacchiar@mat.uniroma2.it}}
\author[b]{\inits{A.}\fnms{Alessandro}~\snm{Pigliacelli}\thanksref{cor1}\ead[label=e2]{alex-matrix23@hotmail.it}}
\thankstext[type=corresp,id=cor1]{Corresponding authors.}
\address[a]{Dept. of Mathematics, \institution{University of Rome ``Tor Vergata''}}
\address[b]{\institution{BIP-Business Integration Partners}}

%\thankstext[id=f1]{}

%\dedicated{}

%\markboth{Authors}{Title}
\markboth{B. Pacchiarotti, A. Pigliacelli}{Large deviations for conditionally Gaussian processes:
 estimates of level crossing probability}

\begin{abstract}
The problem of (pathwise) large deviations for conditionally continuous Gaussian processes is investigated. The theory of large deviations for Gaussian processes is extended to the wider class of random processes -- the conditionally Gaussian processes. The estimates of level crossing probability for such processes are given as an application.
\end{abstract}

\received{\sday{16} \smonth{5} \syear{2018}}% Updated by VTEXPTS2LaTeX.exe, 05.10.2018 09:31
\revised{\sday{30} \smonth{7} \syear{2018}}% Updated by VTEXPTS2LaTeX.exe, 05.10.2018 09:31
\accepted{\sday{1} \smonth{10} \syear{2018}}% Updated by VTEXPTS2LaTeX.exe, 05.10.2018 09:31
\publishedonline{\sday{12} \smonth{10} \syear{2018}}

\end{frontmatter}

\section{Introduction}
In this paper we study some large deviations principles for
conditionally continuous Gaussian processes. Then we find estimates of
level crossing probability for such processes.
Large deviations theory is concerned with the study of
probabilities of very ``rare'' events.
There are events whose probability is very small, however these events
are of great importance; they may represent an \textit{atypical
situation} (i.e. a deviation from the average behavior) that may cause
disastrous consequences: an insurance company or a bank may bankrupt; a
statistical estimator may give a wrong information; a physical or chemical
system may show an atypical configuration.
The aim of this paper is to extend the theory of large deviations for
Gaussian processes to a wider class of random processes -- the
conditionally Gaussian processes. Such processes were introduced in
applications in finance, optimization and control problems. See, for
instance, \cite{DFMR,Loto,Gul} and \cite{AdSaGa}.
 More precisely, Doucet et al. in \cite{DFMR}
considered modelling the behavior
of latent variables in neural networks by Gaussian processes with
random parameters;
Lototsky in \cite{Loto} studied stochastic parabolic equations with
random coefficients; Gulisashvili in \cite{Gul} studied large
deviations principle for some particular stochastic volatility models
where the log-price is, conditionally, a Gaussian process; in \cite
{AdSaGa} probabilities of large extremes of conditionally Gaussian
processes were considered, in particular sub-Gaussian processes i.e.
Gaussian processes with a random variance.
Let $ (Y,Z) $ be a random element on the probability space $(\varOmega,
\cl{F}, \mathbb P)$,\querymark{Q1} where $ Z=(Z_t)_{t \in[0,1]} $ is a process taking
values in $\R$ and $Y$ is an arbitrary random element (a process or a
random variable). We say that $Z$ is a \textit{conditionally Gaussian
process} if the conditional distribution of the process $ Z|Y $ is
(almost surely) Gaussian.
The theory of large deviations for Gaussian processes and for
conditioned Gaussian processes is already well developed. See, for
instance, Section 3.4 in \cite{Deu-Str} (and the references therein)
for Gaussian processes, \cite{Car-Pac-Sal} and \cite{Gio-Pac} for
particular conditioned Gaussian processes. The extension of this theory
is possible thanks to the results obtained by Chaganty in \cite{Cha}.

We consider a family of processes $(Y^n,Z^n)_{n \in\N}$ on a
probability space $(\varOmega, \cl F, \P)$. $ (Y^n)_{n \in\N} $ is a
family of processes taking values in a measurable space $(E_1, \cl
E_1)$ that satisfies a large deviation principle (LDP for short) and
$(Z^n)_{n \in\N}$ is a family of processes taking values in $(E_2,
\cl E_2)$ such that for every $n \in\N$, $Z^n|Y^n$ is a Gaussian
process ($\P$-a.s.). We want to find a LDP for the family $(Z^n)_{n
\in\N}$.

A possible application of LDPs is computing the estimates of level
crossing probability (ruin problem). We will give the asymptotic
behavior (in terms of large deviations) of the probability
\begin{equation*}
p_n=\P \Bigl(\sup_{0\leq t \leq1}\bigl(Z^n_t-
\varphi(t)\bigr)>1 \Bigr),
\end{equation*}
where $\varphi$ is a suitable function.
We will consider the following families of conditionally Gaussian processes.

1) The class of Gaussian processes with random variance and random
mean, i.e.
the processes of the type $(Z_t)_{t \in[0,1]}=(Y_1 X_t+Y_2)_{t \in
[0,1]}$, where $X$ is a centered continuous Gaussian process with
covariance function $k$ and $Y=(Y_1,Y_2)$ is a random element
independent of $X$. Notice that $Z|Y$ is Gaussian with
\[
\Cov(Z_t,Z_s|Y)=\E\bigl[(X_tY_1)
(X_sY_1)|Y\bigr] =Y_1^2k(t,s),
\]
and
\[
\E[Z_t|Y] = \E[Y_1X_t+Y_2|Y]
=Y_2.
\]

2) The class of Ornstein--Uhlenbeck type processes with random diffusion
coefficients. More precisely $(Z_t)_{t \in[0,1]}$ is the solution of
the following stochastic differential equation:
\[
\begin{cases}
dZ_{t} = ( a_{0} + a_{1}Z_{t} )\ dt + Y dW_{t},\quad 0< t\leq1, \\
Z_{0} = x ,
\end{cases}
\]
where $x, a_{0}, a_{1} \in\mathbb{R} $ and $Y$ is a random element
independent of the Brownian motion $(W_t)_{t\in[0,1]}$. %\vadjust{\goodbreak}

The paper is organized as follows. In Section \ref{sect:ld} we recall
some basic facts on large deviations theory for continuous Gaussian processes.
In Section \ref{sect:chaganty} we introduce the conditionally Gaussian
processes and the Chaganty theory. In
Section \ref{sect:main} and \ref{sect:OU} we study the theoretical
problem and we give the main results.
Finally in Section \ref{sect:ruin} we investigate the ruin problem for
such processes.

\section{Large deviations for continuous Gaussian processes} \label{sect:ld}

We briefly recall some main facts on large deviations
principles and reproducing kernel Hilbert spaces for
Gaussian processes we are going to use. For a detailed development of
this very wide theory we can refer, for example, to the following
classical references:
Chapitre II in Azencott \cite{Aze},
Section 3.4 in Deuschel and Strook \cite{Deu-Str}, Chapter 4 (in
particular Sections 4.1 and 4.5) in Dembo and Zeitouni
\cite{Dem-Zei}, for large deviations principles; Chapter 4
(in particular Section 4.3) in \cite
{Hid-Hit}, Chapter 2 (in particular Sections 2.2 and 2.3) in \cite{Ber-Tho}, for
reproducing kernel Hilbert space. Without loss
of generality, we can consider \textit{centered} Gaussian processes.

\subsection{Reproducing kernel Hilbert space}\label
{subsec:rkhs-continuous-gaussian-process}
An important tool to handle continuous Gaussian processes is the
associated reproducing kernel Hilbert space (RKHS).

%which is a Hilbert space of continuous, real valued functions. Since
%the covariance function fully identifies, up to the mean, a Gaussian
%process, we can talk equivalently of RKHS associated with the process
%of with its %covariance function.
Let $U=(U_t)_{t \in[0,1]}$ be a continuous, centered , Gaussian
process on a probability space $(\varOmega,\cl F, \P)$, with covariance
function $k$. From now on, we will denote by $\cl C([0,1])$ the set
of continuous functions on $[0,1]$ endowed with the topology induced by
the sup-norm $(\cl C([0,1]),||\cdot||_{\infty})$. Moreover, we will
denote by $\cl{M}[0,1]$ its dual, i.e. the set of signed Borel
measures on $[0,1]$. The action of $\cl{M}[0,1]$ on $\cl C([0,1])$ is
given by
\[
\langle\lambda, h \rangle= \int_0^1 h(t) \, d
\lambda(t), \quad \lambda\in\cl M[0,1], \ h \in\cl C\bigl([0,1]\bigr).
\]
Consider the set
\[
\cl{L} = \Biggl\{ x \in \cl C\bigl([0,1]\bigr) \; \Big| \; x(t) = \int
_0^1 k(t,s) \, d\lambda(s), \, \lambda\in\cl
M[0,1] \Biggr\}.
\]
The RKHS relative to the kernel $k$ can be constructed as the
completion of the set $\cl{L}$ with respect to a suitable norm.
Consider the set of (real) Gaussian random variables
\[
\varGamma= \Biggl\{ Y \; | \; Y=\langle\lambda, U \rangle= \int
_0^1 U_t \, d\lambda(t) , \,
\lambda\in\cl M[0,1] \Biggr\} \subset L^2(\varOmega, \cl{F},\P).
\]
We have that, for $Y_1, Y_2 \in\varGamma$, say $Y_i= \langle\lambda_i
, U \rangle, \, i=1,2$,
\begin{align}
\label{eqn:inner-product-H} \langle Y_1 , Y_2 \rangle_{L^2(\varOmega, \cl{F},\P)}
&= \Cov \Biggl( \int_0^1 U_t \, d
\lambda_1(t),\int_0^1
U_t \, d\lambda_2(t) \Biggr)
\nonumber
\\*
&= \int_0^1 \int_0^1
k(t,s) \, d\lambda_1(t) d\lambda_2(s).
\end{align}
Define now
\[
H = \overline{\varGamma}^{\|.\|_{L^2(\varOmega,\cl{F},\P)}}.
\]
Since $L^2$-limits of Gaussian random variables are still Gaussian, we
have that $H$ is a closed subspace of $L^2(\varOmega,\cl{F},\P)$
consisting of real Gaussian random variables. Moreover, it becomes a
Hilbert space when endowed with the inner product
\[
\langle Y_1, Y_2 \rangle_H = \langle
Y_1, Y_2 \rangle_{L^2(\varOmega
,\cl{F},\P)}, \quad
Y_1,Y_2 \in H.
\]
\begin{remark}\rm
We remark that, since any signed Borel measure $\lambda$ can be weakly
approximated by a linear combination of Dirac deltas, the Hilbert
space $H$ above is nothing but the Hilbert space generated by the
Gaussian process $U$, namely
\begin{align*}
H &= \overline{\mathit{sp}\bigl\{U_t,\, t \in[0,1]\bigr\}}^{\|.\|_{L^2(\varOmega,\cl
{F},\P)}}\\
&= \overline{ \Biggl\{ \sum_{j=1}^n
a_j\,U_{t_j} \; | \; n \in \mathbb{N}, a_j \in
\mathbb{R}, t_j \in[0,1] \Biggr\} }^{\|.\|
_{L^2(\varOmega,\cl{F},\P)}}.
\end{align*}
\end{remark}

Consider now the following mapping
\begin{align}
\label{eqn:loeve-isometry} \cl{S} : H &\longrightarrow\cl C\bigl([0,1]\bigr), \quad\quad Y
\mapsto{(\cl{S}Y)}_. = \E(U_. Y).
\end{align}

\begin{definition} \label{definition:rkhs2}
Let $U=(U_t)_{t \in[0,1]}$ be a continuous Gaussian process. We define
the reproducing kernel Hilbert space relative to the Gaussian process
$U$ as
\[
\cl{H} = \cl{S}(H) = \bigl\{ h \in\cl C\bigl([0,1]\bigr) \; | \; h(t) = {(\cl
{S}Y)}_t, \, Y \in H \bigr\},
\]
with an inner product defined as
\[
\langle h_1 , h_2 \rangle_{\cl{H}} = \bigl
\langle\cl{S}^{-1} h_1, \cl {S}^{-1}
h_2 \bigr\rangle_H = \bigl\langle\cl{S}^{-1}
h_1, \cl{S}^{-1} h_2 \bigr\rangle_{L^2(\varOmega,\cl{F},\P)},
\quad h_1,h_2 \in\cl{H}.
\]
\end{definition}
Then, we have
\begin{lemma} \label{lemma:loeve} (Theorem 35 in \cite{Ber-Tho}). Let
$H$ be the Hilbert space of the continuous Gaussian process $U$ defined
above. Then $H$ is isometrically isomorphic to the Reproducing Kernel
Hilbert Space $\cl{H}$ of $U$, and the corresponding isometry is given
by (\ref{eqn:loeve-isometry}).
\end{lemma}

The map $\cl{S}$ defined in (\ref{eqn:loeve-isometry}) is referred to
as \textit{Lo\`{e}ve isometry}. Since the covariance function fully
identifies, up to the mean, a Gaussian process, we can talk
equivalently of RKHS associated with the process or with its covariance
function.

\subsection{Large deviations}

\begin{definition} \label{definition:ldp}(LDP)
Let $E$ be a topological space, $\cl{B}(E)$ the Borel $\sigma
$-algebra and $(\mu_n)_{n\in\N}$ a family of probability measures on
$\cl{B}(E)$; let $\gamma\, : \N\rightarrow\mathbb{R}^+ $ be a
function, such that $\gamma(n) \rightarrow+\infty$ as $n\to+\infty
$. We say that the family of probability measures $(\mu_n)_{n\in\N}$
satisfies a large deviation principle (LDP) on $E$ with the rate
function $I$ and the speed $\gamma(n)$ if, for any open set
$\varTheta$,
\[
-\inf_{x \in{\varTheta} } I(x) \le\liminf_{n\to+\infty}
\frac{1}{\gamma(n)} \log\mu_n (\varTheta)
\]
and for any closed set $\varGamma$
\begin{equation}
\label{eq:upperbound} \limsup_{n\to+\infty}\frac{1}{
\gamma(n)} \log
\mu_n (\varGamma) \le-\inf_{x \in{\varGamma}} I(x).
\end{equation}
\end{definition}
A rate function is a lower semicontinuous mapping $I:E\rightarrow
[0,+\infty]$. A rate function $I$ is said \textit{good} if the sets
$\{I\le a\}$ are compact for every $a \ge0$.
\begin{definition} \label{definition:wldp}(WLDP)
Let $E$ be a topological space, $\cl{B}(E)$ the Borel $\sigma
$-algebra and $(\mu_n)_{n\in\N}$ a family of probability measures on
$\cl{B}(E)$; let $\gamma\, : \N\rightarrow\mathbb{R}^+ $ be a
function, such that $\gamma(n) \rightarrow+\infty$ as $n\to+\infty
$. We say that the family of probability measures $(\mu_n)_{n\in\N}$
satisfies a weak large deviation principle (WLDP) on $E$ with the rate
function $I$ and the speed $\gamma(n)$ if the upper bound (\ref
{eq:upperbound}) holds for compact sets.
\end{definition}

\begin{remark}\rm
We say that a family of continuous processes $((X_t^n)_{t \in
[0,1]})_{n\in\N}$ satisfies a LDP if the associated family of laws
satisfy a LDP on $\cl C([0,1])$.
\end{remark}

The following remarkable theorem (Proposition 1.5 in \cite{Aze}) gives
an explicit expression of the Cram\'er transform $\varLambda^*$ of a
continuous centered Gaussian process $(U_t)_{t \in[0,1]}$ with
covariance function $k$. Let us recall that
\[
\varLambda(\lambda)=\log\E\bigl[\exp\bigl(\langle U, \lambda\rangle\bigr)\bigr]=
\frac
{1}{2} \int_0^1 \int
_0^1 k(t,s) \, d\lambda(t) d\lambda(s),
\]
for $\lambda\in\cl M[0,1]$.

\begin{theorem} \label{theorem:cramer-transform}
Let $(U_t)_{t \in[0,1]}$ be a continuous, centered Gaussian process
with covariance function $k$. Let $\varLambda^*$ denote the Cram\'{e}r
transform of $\varLambda$, that is,
\begin{align*}
\varLambda^*(x)& = \sup_{\lambda\in\cl M[0,1]} \bigl( \langle \lambda, x
\rangle- \varLambda(\lambda) \bigr)
\\
&= \sup_{\lambda\in\cl M[0,1]} \left( \langle\lambda, x \rangle-
\frac{1}{2} \int_0^1 \int
_0^1 k(t,s) \, d\lambda(t) d\lambda(s) \right).
\end{align*}
Then,
\begin{equation}
\label{eqn:cramer-transform-gaussian} \varLambda^*(x) = %
\begin{cases} \frac{1}{2} \|x \|_{\cl{H}}^2, & x \in\cl{H}, \\
+\infty & \text{otherwise},
\end{cases} %
\end{equation}
where $\cl{H}$ and $\| . \|_{\cl{H}}$ denote, respectively, the
reproducing kernel Hilbert space and the related norm associated to
the covariance function $k$.
\end{theorem}

In order to state a large deviation principle for a family of Gaussian
processes, we need the following definition.
\begin{definition}
A family of continuous processes ${((X^n_t)_{t \in[0,1]}})_{n\in\N}$
is exponentially tight at the speed $\gamma(n)$ if, for every $R>0$
there exists a compact set $K_R$ such that
\begin{equation}
\label{definition:exponential-tightness} \limsup_{n\to+\infty} \frac{1}{\gamma(n)} \log\P
\bigl(X^n \notin K_R\bigr) \le-R.
\end{equation}
\end{definition}

If the means and the covariance functions of an exponentially tight
family of Gaussian processes have a good limit behavior, then the
family satisfies a large deviation principle, as stated in the
following theorem
which is a consequence of the classic abstract G\"{a}rtner--Ellis Theorem (Baldi Theorem 4.5.20 and Corollary 4.6.14 in
\cite{Dem-Zei}) and Theorem \ref{theorem:cramer-transform}.
\begin{theorem}\label{theorem:ldp-gaussian}
Let $((X_t^n)_{t \in[0,1]})_{n\in\N}$ be an exponentially tight
family of continuous Gaussian processes with respect to the speed
function $\gamma(n)$. Suppose that, for any $\lambda\in\cl M[0,1]$,
\begin{equation}
\label{eq:meanlimit} \lim_{n\to+\infty} \E \bigl[ \bigl\langle\lambda,
X^n \bigr\rangle \bigr] = 0
\end{equation}
and the limit
\begin{equation}
\label{eq:covlimit} \varLambda(\lambda) = \lim_{n\to+\infty} \gamma(n) \Var
\bigl( \bigl\langle\lambda, X^n \bigr\rangle \bigr) = \int
_0^1 \int_0^1
\bar {k}(t,s) \, d\lambda(t) d\lambda(s)
\end{equation}
exists for some continuous, symmetric, positive definite function
$\bar{k}$, which is the covariance function of a continuous Gaussian process.
Then $((X_t^n)_{t \in[0,1]})_{n\in\N}$
satisfies a large deviation principle on $\cl C([0,1]))$, with the
speed $\gamma(n)$ and the good rate function
\begin{equation}
I(h) = %
\begin{cases} \frac{1}{2}  \llVert  h  \rrVert _{\bar{\cl H }}, & h
\in\bar{\cl H },\\ +\infty& \text{otherwise},
\end{cases} %
\end{equation}
where ${\bar{\cl H}}$ and $ \llVert  .  \rrVert _{\bar{\cl H }}$
respectively denote the reproducing kernel Hilbert space and the
related norm associated to the covariance function $\bar{k}$.
\end{theorem}

A useful result which can help in investigating the exponential
tightness of a family of continuous centered Gaussian processes is the
following proposition (Proposition 2.1 in \cite{Mac-Pac}); the
required property follows from H\"older continuity of the mean and the
covariance function.
\begin{proposition} \label{proposition:cond-exp-tight}
Let $((X_t^n)_{t \in[0,1]})_{n\in\N}$ be a family of continuous
Gaussian processes with $X_0^n = 0$ for all $n\in\N$. Denote
$m^n(t) = \E[X_t^n]$ and  $k^n(t,s) = \Cov(X_t^n, X_s^n)$.
Suppose there exist constants $\alpha, M_1, M_2 > 0$ such that for
$n\in\N$
\[
\sup_{s,t \in[0,1], \, s \ne t} \frac{|m^n(t) -
m^n(s)|}{|t-s|^\alpha} \le M_1
\]
and
\begin{equation}
\label{eqn:cov-tight} \sup_{s,t \in[0,1], \, s \ne t} \gamma(n)\frac
{|k^n(t,t)+k^n(s,s)-2k^n(s,t)|}{ |t-s|^{2\alpha}} \le
M_2.
\end{equation}
Then the family $((X_t^n)_{t \in[0,1]})_{n\in\N}$ is exponentially
tight with respect to the speed function $\gamma(n)$.
\end{proposition}

\section{Conditionally Gaussian processes}\label{sect:chaganty}
In this section we introduce conditionally Gaussian processes and the
Chaganty theorem which allows us to find a LDP for families of such
processes.  We also recall, for sake of completeness,
some results about conditional distributions in Polish spaces. We
referred to Section 3.1 in \cite{Bor} and Section 4.3 in \cite{Bal}.

Let $Y$ and $Z$ two random variables, defined on the same probability
space\break $(\varOmega, \cl F, \P)$, with values in the measurable spaces $
(E_1, \cl E_1)$ and $(E_2,\cl E_2)$ respectively, and let us denote by
$\mu_1,\mu_2$ the (marginal) laws of
$Y$ and $Z$ respectively and by $\mu$ the joint distribution of $(Y,Z)$
on $(E, \cl E)=(E_1\times E_2, \cl E_1 \times\cl E_2)$. A family of
probabilities $(\mu_2(dz|y)))_{y\in E_1}$ on $(E_2,\cl E_2)$ is a
regular version of the conditional law of $Z$ \it given \rm$Y$ if\looseness=-1
\begin{enumerate}
\item For every $B\in\cl E_2$, the map $y\mapsto\mu_2(B|y)$ is $\cl
E_1$-measurable.
\item For every $B\in\cl E_2$ and $A\in\cl E_1$, $\P(Y\in A,Z\in
B)=\int_A \mu_2(B|y) \mu_1(dy)$.
\end{enumerate}
%For every continuous and bounded function $f:E\to\R$,
%
In this case we have
\[
\mu(dy,dz)= \mu_2(dz|y)\mu_1(dy).
\]
%
%Some manipulations (see, for example Section 4.3 in \cite{Bal}) show
%that if $y\mapsto\mu_2(A|y)$ is $\cl G$-measurable
%then $(\mu_2(dz|y))_{y\in G}$
%is a regular conditional distribution of $Z$ given $Y$
%(or of $\mu_2$ given $\mu_1$) if for every bounded
%measurable $f:E\to\R$, $\P$-almost sure we have,
%$$
%\E[f(Z)|Y]= \int_E f(z) \mu_2(dz|Y).
%$$
In this section we will use the notation $(E, \cl{B})$ to indicate a
Polish space (i.e. a complete separable metric space) with the Borel
$\sigma$-field, and we say that a sequence $(x_n)_{n\in\N} \subset
\E$ converges to $x \in E$, $x_n\rightarrow x$, if $d_E
(x_n,x)\rightarrow0$, as $n\rightarrow\infty$, where $d_E$ denotes
the metric on $E$. Regular conditional probabilities do not always exist,
but they exist in many cases.
The following result, that immediately follows from Corollary 3.2.1 in
\cite{Bor}, shows that in Polish spaces the regular version of the
conditional probability is well defined.

%Let $(E_1, \cl{B}_1)$ and $(E_2, \cl{B}_2)$ be two Polish spaces.

%We denote with $(\mu_n)_{n \in\N}$ a sequence of probabilities
%measures on $(E, \cl{B})=(E_1\times E_2, \cl{B}_1 \times\cl{B}_2)$
%(the sequence of \textit{joint distributions}), with
%$\left( \mu_{1n}\right)_{n \in\N} $ the sequence of the
%\textit{marginal distributions} on $ (E_1, \cl{B}_1)$ and with $(
%\mu_{2n}(\cdot| x_{1} ))_{n \in\N}$ the sequence of
%\textit{conditional distributions}
%on $(E_2, \cl{B}_2)$ ($x_1 \in\varOmega_1,$), i.e.
%$$\mu_{n}(B_{1} \times B_{2})=\int_{B_1}\mu_{2n}(B_{2} | x_{1} )\ d
%\mu_{1n}(x_{1}).$$

\begin{proposition} \label{prop:condprob}
Let $(E_1, \cl{B}_1)$ and $(E_2, \cl{B}_2)$ be two Polish spaces
endowed with their Borel
$\sigma$-fields, $\mu$ be a probability measure on $(E, \cl
{B})=(E_1\times E_2, \cl{B}_1 \times\cl{B}_2)$. Let $\mu_i$ be the
marginal probability measure on $(E_i, \cl{B}_i)$, $i=1,2$. Then there
exists $\mu_1$-almost sure a unique
regular version of the conditional law of $\mu_2$ given $\mu_1$, i.e.
\[
\mu(dy,dz)= \mu_2(dz|y)\mu_1(dy).
\]
\end{proposition}
In what follows we always suppose random variables taking values in
a Polish space.
\begin{definition}
Let $ (Y,Z) $ be a random element on the probability space $(\varOmega,
\cl{F}, \P)$, where $ Z=(Z_t)_{t \in[0,1]} $ is a real process and
$Y$ is an arbitrary random element (a process or a random variable). We
say that $Z$ is a \textit{conditionally Gaussian process} if the
conditional distribution of the process $ Z|Y $ is (almost surely) Gaussian.
We denote by $(Z^y_t)_{t\in[0,1]}$ the Gaussian process $ Z|Y=y $.
\end{definition}

%\begin{example}\rm
% Let $W_1,\ W_2$ be mutually independent Brownian motions on the
%probability space $(\varOmega, \cl{F}, \P)$ and let $ (\varTheta_t,
%\varPsi_t)_{t \in[0,T]} $ be the diffusion process defined by
% \begin{equation}
% \left\{
% \begin{aligned}
% &d\varTheta_{t} = \left( a_{0}(\varPsi_t,t) + a_{1}(\varPsi_t,t)\varTheta_{t}
%\right)\ dt + b(\varPsi_t,t) \ dW^1_{t} \\
% &d\varPsi_t = \left( A(\varPsi_t,t)\right)\ dt + B(\varPsi_t,t) \ dW^2_{t} ,
%\end{aligned} \right.
% \end{equation}
% for $0<t\leq T$, $\varTheta_{0} = x_1 $ and $\varPsi_0 =x_2$. If the
%coefficients are regular enough
% then the process $ (\varTheta_t)_{t \in[0,T]} $ is conditionally
%Gaussian. For further details, see \cite{Lip-Shi2}.
%\end{example}
The main tool that we will use to study LDP for a family of conditionally
Gaussian processes is provided by Chaganty Theorem (Theorem 2.3 in
\cite{Cha}).
Let $(E_1, \cl{B}_1)$ and $(E_2, \cl{B}_2)$ be two Polish spaces.
We denote by $(\mu_n)_{n \in\N}$ a sequence of probabilities
measures on $(E, \cl{B})=(E_1\times E_2, \cl{B}_1 \times\cl{B}_2)$
(the sequence of \textit{joint distributions}), by
$ ( \mu_{1n} )_{n \in\N} $ the sequence of the \textit
{marginal distributions} on $ (E_1, \cl{B}_1)$ and by $(\mu
_{2n}(\cdot| x_{1} ))_{n \in\N}$ the sequence of \textit
{conditional distributions}
on $(E_2, \cl{B}_2)$ ($x_1 \in E_1$,), given by Proposition \ref
{prop:condprob}, i.e.
\[
\mu_{n}(B_{1} \times B_{2})=\int
_{B_1}\mu_{2n}(B_{2} | x_{1} )\
\mu _{1n}(dx_{1}).
\]

\begin{definition}\label{def: LDP continuity condition} Let $(E_1, \cl
{B}_1)$, $(E_2, \cl{B}_2)$ be two Polish spaces and $x_{1} \in E_{1}$.
We say that the sequence of conditional laws $ ( \mu_{2n}(\cdot|
x_{1} )  )_{n \in\N} $ on $(E_2,\cl{B}_2)$ satisfies the LDP
continuously in $x_{1}$ with the rate function $J(\cdot|x_{1})$ and
the speed $\gamma(n)$, or simply, the LDP continuity condition holds, if
\begin{enumerate}
\item[a)] For each $x_{1} \in E_{1},J(\cdot| x_{1})$ is a good rate
function on $E_{2}$.
\item[b)] For any sequence $ ( x_{1n} )_{n \in\N} \subset
E_{1}$ such that $x_{1n} \rightarrow x_{1}$, the sequence of measures
$ ( \mu_{2n}(\cdot| x_{1n} ) )_{n \in\N} $ satisfies a
LDP on $E_{2}$ with the (same) rate function $J(\cdot|x_{1})$ and the
speed $\gamma(n)$.
\item[c)] $J(\cdot| \cdot)$ is lower semicontinuous as a function of
$(x_{1}, x_{2}) \in E_1 \times E_2$.\vadjust{\goodbreak}
\end{enumerate}
\end{definition}

\begin{theorem}[Theorem 2.3 in \cite{Cha}]\label{theorem:chaganty}
%Let $ (E_{1}, \cl{B}_{1}), (E_{2}, \cl{B}_{2})$ be two Polish spaces
%with their associated Borel $\sigma$-fields.
Let $(E_1, \cl{B}_1)$, $(E_2, \cl{B}_2)$ be two Polish spaces. For
$i=1,2$ let $ ( \mu_{\mathit{in}} )_{n \in\N}$ be a sequence of
measures on $ (E_{i}, \cl{B}_{i})$. For $x_1 \in E_1$, let $ (
\mu_{2n}(\cdot| x_{1} ) )_{n \in\N} $ be the sequence of the
conditional laws (of $\mu_{2n}$ given $\mu_{1n}$) on $(E_{2}, \cl
{B}_{2})$. Suppose that the following two conditions are satisfied:
\begin{enumerate}
\item[i)] $ ( \mu_{1n} )_{n \in\N} $ satisfies a LDP on
$E_1$ with the good rate function $I_{1}(\cdot)$ and the speed $
\gamma(n)$.
\item[ii)] For every $x_1 \in E_1$, the sequence $ ( \mu
_{2n}(\cdot| x_{1} ) )_{n \in\N} $ satisfies the LDP
continuity condition on $E_2$ with the rate function $J( \cdot|x_{1})$
and the speed $ \gamma(n)$.
\end{enumerate}
Then the sequence of joint distributions $ ( \mu_{n} )_{n
\in\N}$ satisfies a WLDP on $E=E_1\times E_2$ with the speed $ \gamma
(n) $ and the rate function
\[
I(x_{1},x_{2}) = I_{1}(x_{1}) +
J(x_{2}| x_{1}),\quad x_1 \in E_1,
\, x_2 \in E_2.
\]
The sequence of marginal distributions $ ( \mu_{2n} )_{n
\in\N} $ defined on $ (E_{2}, \cl{B}_{2})$, satisfies a LDP with
the speed $ \gamma(n)$ and the rate function
\[
I_{2}(x_{2})= \inf_{x_{1} \in E_{1} }I(x_{1},x_{2}).
\]
Moreover, if $I(\cdot,\cdot)$ is a good rate function then $ (
\mu_{n} )_{n \in\N} $ satisfies a LDP and $I_{2}(\cdot) $ is
a good rate function.
\end{theorem}

\section{Gaussian process with random mean and random variance} \label
{sect:main}
Let $\alpha> 0$ and define $\cl C_{\alpha}([0,1]) =\{y\in\cl
{C}([0,1]) : y(t)\geq\alpha, \,\, t\in[0,1]\}$ (with the uniform
norm on compact sets). $\cl C_{\alpha}([0,1])$ is a Polish space.
Consider the family of processes $(Y^n,Z^n)_{n \in\N}$, where
$(Y^n)_{n \in\N}=(Y_1^n,Y_2^n)_{n \in\N}$ is a family of processes
with paths in $\cl C_{\alpha}([0,1])\times\cl{C}([0,1])$ and for $n
\in\N$, $Z^n=X^nY^n_1 + Y_2^n$ with $(Y^n)_{n \in\N}$ independent
of $(X^n)_{n \in\N}$. Suppose $ ((X^n_t)_{t \in[0,1]})_{n \in\N} $
is family of continuous centered Gaussian processes which satisfy the
hypotheses of Theorem \ref{theorem:ldp-gaussian} and
suppose that $(Y^n)_{n \in\N}$ satisfies a LDP with the good rate
function $I_{Y}$ and the speed $\gamma(n)$. We want to prove a LDP
principle for $(Z^n)_{n \in\N}$.

%{\color{red}\begin{example}[Sub-Gaussian]\label{ex:sub}\rm A process
%$(Z_t)_{t\geq0}$ is called sub-Gaussian if it can be written in the
%form
%$Z_t=A^{1/2} X_t$, where $A$ is a stable random variable and $(XZ_t)_{t
%\geq0}$ is a Gaussian process independent of $A$. This is an example
%of Gaussian process with random variance.
%\end{example}}

%%%%%%%%%%%%%%%%%%%%%%%%%%%%%%%%%%%%%%%%%%%%%%%%%%%
%
\begin{proposition}\rm\label{proposition:Xy1+ y2}
Let $(( X_{t}^n)_{t \in[0,1]})_{n\in\N}$ be a family of continuous
Gaussian processes which satisfies the hypotheses of Theorem \ref
{theorem:ldp-gaussian}
and let $y=(y_1,y_2) \in\cl C_{\alpha}([0,1])\times\cl{C}([0,1])$. Then
the family $((X^n_{t}y_1(t)+y_2(t) )_{t \in[0,1]})_{n\in\N}$ is
still a family of
continuous Gaussian processes which satisfies the hypotheses of Theorem
\ref{theorem:ldp-gaussian} with the same speed function and limit
covariance function (depending only on $y_1$) $k^{y_1}$ given by
\begin{equation}
\label{eq:ky} k^{y_1}(s,t)= y_1(s)y_1(t)\bar
k(s,t).
\end{equation}

Therefore, also $((X^n_{t}y_1(t)+y_2(t))_{t \in[0,1]})_{n\in\N}$
satisfies a LDP with the good rate function
\begin{equation}
\label{def: rate function p. gaussiano cambio RKHS} \varLambda^{*}_y(z) =
\begin{cases}
\frac{1}{2}\lVert z\rVert^{2}_{\bar{\cl{H}}_{y_1}}, & z
\in\bar {{\cl{H}}_{y_1} },\\
+\infty& \text{otherwise},\ \end{cases} %
 =
\begin{cases}
\frac{1}{2} \lVert\frac{z-y_2}{y_1} \rVert^{2}_{\bar{\cl
{H}}}, & \frac{z-y_2}{y_1} \in\bar{\cl{H}},
\\
+\infty& \text{otherwise},
\end{cases}
\end{equation}
where $\bar{\cl{H}}_{{y_1}}$ is the RKHS associated to the covariance
function defined in (\ref{eq:ky}).
\end{proposition}

\begin{proof}
This is a simple application of the contraction principle.\vadjust{\goodbreak}
\end{proof}

\begin{remark}\rm\label{remark: cambio RKHS}
If $y_1(t)=y_1> 0$ for all $t \in[0,1]$, then we have
\begin{equation*}
\varLambda^{*}_y(z)=
\begin{cases}
\frac{1}{2y_1^2} \lVert z -y_2 \rVert^{2}_{\cl{H}},
& z-y_2 \in\bar{\cl{H}}, \\
+\infty& \text{otherwise}.
\end{cases}
\end{equation*}
\end{remark}

\begin{definition}
Let $ (E, d_{E}) $ be a metric space, and let $(\mu_n)_{n\in\N}$,
$(\tilde{\mu}_{n})_{n\in\N}$ be two families of probability
measures on $ E $. Then $(\mu_n)_{n\in\N}$ and $(\tilde{\mu
}_n)_{n\in\N}$ are exponentially equivalent (at the speed $\gamma
(n)$) if there exist a family of probability spaces $ ((\varOmega,\cl
{F}^n, \P^n))_{n\in\N} $ and two families of $ E $-valued random
variables $ (Z^n)_{n\in\N} $ and $ (\tilde{Z}^n)_{n\in\N} $ such
that, for any $ \delta> 0 $, the set $ \lbrace\omega: d_E(\tilde
{Z}^n(\omega),Z^n(\omega))>\delta\rbrace$
is $ \cl{F}^n $-measurable and
\[
\limsup_{n\to+\infty} \gamma(n) \log\P^n
\bigl(d_E\bigl(\tilde{Z}^n(\omega ),Z^n(
\omega)\bigr)>\delta\bigr) = -\infty.
\]
\end{definition}

As far as the LDP is concerned exponentially equivalent measures are
indistinguishable. See Theorem 4.2.13 in \cite{Dem-Zei}.

\begin{proposition}\rm\label{proposition:expeq}
Let $(( X_{t}^n)_{t \in[0,1]})_{n\in\N}$ be an exponential tight (at
the speed $\gamma(n)$) family of continuous Gaussian processes.
Let $(y^n)_{n\in\N} \subset\cl C_{\alpha}([0,1])\times\cl{C}([0,1])
$ such that $y^n \to y$ in $\cl C_{\alpha}([0,1])\times\cl{C}([0,1])$.
Then, the family of processes $ (( y_2^n(t)+ y_1^n(t)X^n(t)))_{t\in
[0,1]})_{n\in\N}$ is exponentially equivalent to
$(( y_2(t) + y_1(t)X^n(t)))_{n\in\N}$.
\end{proposition}

\begin{proof}
Let $Z^n(t)=y_2(t)+ y_1(t)X^n(t)$ and $\tilde Z^n(t)=y_2^n(t) +
y_1^n(t)X^n(t)$ for $t\in[0,1]$, $n\in\N$.
Then, for any $\delta>0$,
\[
\P\bigl(\big\| Z^n - \tilde Z^n\big\|_{\infty}> \delta
\bigr)\leq\P \biggl( \big\|X^n\big\|_{\infty}\, \big\|y_1^n-y_1\big\|_{\infty}>
\frac{\delta}2 \biggr) + \P \biggl( \big\|y_2^n-y_2\big\|_{\infty}>
\frac{\delta}2 \biggr).
\]
For $n$ large enough $||y_2^n-y_2||_{\infty}\leq\frac{\delta}2$ and
thanks to (\ref{definition:exponential-tightness})
\[
\limsup_{n\to+\infty} \frac{1}{\gamma(n)}\log\P \biggl(
\big\|X^n\big\|_{\infty} >\frac{\delta}{ 2\|y_1^n-y_1\|_{\infty}} \biggr)=-\infty,
\]
therefore
\[
\limsup_{n\to+\infty}\frac{1}{\gamma(n)}\log\P\bigl(\big\| Z^n -
\tilde Z^n\big\|_{\infty}> \delta\bigr)=-\infty.\qedhere
\]
\end{proof}

Let us denote $J(z|y)=\varLambda^*_y(z) $, for $z \in\cl{C}([0,1])$ and
$y\in\cl C_{\alpha}([0,1])\times\cl{C}([0,1])$. We want to prove the
lower semicontinuity of $J(\cdot|\cdot)$.

\begin{proposition} \label{proposition:lowsemicon} If
$(z^n,y^n)\rightarrow(z ,y )$ in $\cl{C}([0,1])\times\cl C_{\alpha}
([0,1])\times\cl{C}([0,1])$, then
\[
\liminf_{n\to+\infty}J\bigl(z^n|y^n\bigr) \geq
J(z |y ).
\]
\end{proposition}
\begin{proof}
Thanks to the lower semicontinuity of $ \lVert\cdot\rVert
^{2}_{\bar{\cl{H}}} $
\begin{align*}
\liminf_{(y^n,z^n)\rightarrow(y ,z )} J
\bigl(z^n|y^n\bigr)& =\liminf_{(y^n,z^n)\rightarrow(y ,z )}
\frac{1}{2} \biggl\lVert\frac{z^n-
y_2^n}{y^n_1} \biggr\rVert^{2}_{\bar{\cl{H}}}
\\
&=\liminf_{h^n\rightarrow h } \frac{1}{2} \bigl\lVert h^n
\bigr\rVert^{2}_{\bar{\cl{H}}} \geq\frac{1}{2} \lVert h
\rVert^{2}_{\bar{\cl{H}}} = \frac{1}{2} \biggl\lVert
\frac{z - y_2
}{y _1} \biggr\rVert^{2}_{\bar{\cl{H}}}= J(x |y )
\end{align*}
where, $h^n=\frac{z^n- y^n_2}{y^n_1} \overset{\cl
{C}([0,1])}{\longrightarrow} h =\frac{z - y _2}{y _1}$.
\end{proof}

\begin{theorem}
Consider the family of processes $(Y^n,Z^n)_{n \in\N}$, where
$(Y^n)_{n \in\N}=\break(Y_1^n,Y_2^n)_{n \in\N}$ is a family of processes
with paths in $\cl C_{\alpha}([0,1])\times\cl{C}([0,1])$ and for $n
\in\N$, $Z^n=X^nY^n_1 + Y_2^n$ with $(Y^n)_{n \in\N}$ independent
of $(X^n)_{n \in\N}$. Suppose $ ((X^n_t)_{t \in[0,1]})_{n \in\N} $
is family of continuous centered Gaussian processes which satisfy the
hypotheses of Theorem \ref{theorem:ldp-gaussian} and
suppose that $(Y^n)_{n \in\N}$ satisfies a LDP with the good rate
function $I_{Y}$ and the speed $\gamma(n)$. Then $(Y^n,Z^n)_{n \in\N
}$ satisfies the WLDP with the speed $\gamma(n)$ and the rate function
\[
I(y,z)= I_Y(y) + J(z|y),
\]
and $(Z^n)_{n \in\N}$ satisfies the LDP with the speed $\gamma(n)$
and the rate function
\[
I_Z(z)= \inf_{y\in\cl C_{\alpha}([0,1]) }\bigl\{I_Y(y) +
J(z|y)\bigr\}.
\]
\end{theorem}

\begin{proof}
Thanks to Propositions \ref{proposition:Xy1+ y2}, \ref
{proposition:expeq} and \ref{proposition:lowsemicon} the family of
processes $(Y^n,Z^n)_{n \in\N}$ satisfies the
hypotheses of Theorem \ref{theorem:chaganty}, therefore the theorem
holds.
\end{proof}

\section{Ornstein--Uhlenbeck processes with random diffusion coefficient} \label{sect:OU}
Let $\alpha> 0$ and let again $\cl C_{\alpha}([0,1]) =\{y\in\cl
{C}([0,1]) : y(t)\geq\alpha, \,\, t\in[0,1]\}$. Consider the family
of processes $(Y^n,Z^n)_{n \in\N}$, where $(Y^n)_{n \in\N}$ is a
family of processes with paths in $\cl C_{\alpha}([0,1])$ and for $n \in
\N$, $Z^n$ is the solution of the following stochastic differential equation.
\begin{equation}
\label{eq:OU}
\begin{cases}
dZ_{t}^{n} = \bigl( a_{0} + a_{1}Z_{t}^{n}
\bigr)\ dt + \frac{1}{\sqrt{n}}Y_{t}^{n}\ dW_{t}
\quad0< t\leq1,\\
Z_{0}^{n} = x,
\end{cases}
\end{equation}
where, $x,\, a_{0},\, a_{1} \in\mathbb{R} $ and $(Y^n)_{n \in\N}$
is a family of random processes independent of the Brownian motion
$(W_t)_{t\in[0,1]}$.

Suppose that $(Y^n)_{n \in\N}$ satisfies a LDP with the good rate
function $I_{Y}$, and the speed $\gamma(n)=n$. We want to prove a LDP
principle for $(Z^n)_{n \in\N}$.

%{\color{red}\begin{example}[Fractional stochastic volatility model]
%\label{ex:frac}\rm In some fractional models the log of the asset
%price is
%$$dS_t= S_t \sigma(\hat B_t) d(\sqrt{1-\rho^2} W^1_t + \rho W^2t),
%\quad S_0=s_0>0, \quad0\leq t\leq T,$$
%where $s_0$ is the initial price, while $T>0$ is the time horizon. The
%process $W^1$ and $W^2$ are independent standard Brownian motions, and
%$\rho\in(-1,1)$ is the correlation coefficient.
%$\hat B$ is a volterra self similar Gaussian process. In this model
%the process $(Z_t)_{t\geq0}$, $Z_t=\log S_t$, is a Ornstein-Uhlenbeck
%with random coefficients.
%\end{example}}
Let $Z^{n,y}$, $y\in\cl C_{\alpha}([0,1])$, be the solution of the
following stochastic differential equation,
\begin{equation}
\label{eq:OUy} %
\begin{cases}
dZ_{t}^{n,y} =  ( a_{0} + a_{1}Z_{t}^{n,y} )\ dt + \frac{1}{\sqrt{n}}y(t)\  dW_{t}, \quad 0< t\leq1,\\
Z_0^{n,y} =x,
\end{cases} %
\end{equation}
that is
\begin{align}
 Z_{t}^{n,y}
&= e^{a_{1}t}  \Biggl( x + \frac{a_0}{a_1}\bigl[ 1 - e^{-a_1
t}\bigr] + \frac{1} {\sqrt{n}} \int_{0}^{t} e^{-a_{1}s}y(s) dW_{s} \Biggr)\nonumber\\*
&= m(t)+ e^{a_{1}t}\frac{1} {\sqrt{n}} \int_{0}^{t} e^{-a_{1}s}y(s) dW_{s},
\label{eq:explicit-OU} %
\end{align}
$m(t)=e^{a_{1}t}(x + \frac{a_0}{a_1}[ 1 - e^{-a_1 t}])$.
$((Z_{t}^{n,y})_{t\in[0,1]})_{n\in\N}$ is a family of Gaussian
processes and a family of diffusions.
It is well known from the Wentzell--Friedlin theory that
$((Z_{t}^{n,y})_{t\in[0,1]})_{n\in\N}$
satisfies the LDP in $\cl{C}([0,1])$ with the speed $\gamma(n)=n$ and
the good rate function\vadjust{\goodbreak}
\begin{equation}
\label{eq:rate function_F-W_xi_sigma} J(f|y)=
\begin{cases}
\frac{1}{2}\int_{0}^{1} ( \frac{\dot
{f}(t)-(a_{0}+a_{1}f(t))}{y(t)})^{2}dt, & f\in H_{1}^{x},\\
+ \infty, & f\notin H_{1}^{x},
\end{cases}
\end{equation}
where
\[
H_{1}^{x}:= \Biggl\lbrace f : f(t) = x + \int
_{0}^{t} \phi(s) \ ds,\; \phi\in L^{2}
\bigl([0,1]\bigr) \Biggr\rbrace.
\]
And it is well known from the theory of Gaussian processes that the
family\break $((Z_{t}^{n,y})_{t\in[0,1]})_{n\in\N}$
satisfies the LDP in $\cl{C}([0,1])$ with the speed $\gamma(n)=n$ and
the good rate function
\begin{equation}
\label{eq:rate function_GP} %
J(f|y) =
\begin{cases}
\frac{1}{2}\lVert f-{m} \rVert^{2}_{\cl{H}_y},
& f-{m}\in\cl {H}_y,\\
+ \infty, & f-{m}\notin\cl{H}_y,
\end{cases}
\end{equation}
where $\cl H_y$ is the reproducing kernel Hilbert space associate to the
covariance function
\[
k^y(s,t)= e^{a_{1}(s+t)} \int_{0}^{s\wedge t}
e^{-2a_{1}u}y^{2}(u)\, du.
\]

The two rate functions (for the unicity of the rate function) are the
same rate function. So we can deduce a LDP for the family $(Z^n)_{n\in
\N}$ in two different ways.
First let $(Z_{t}^{n,y})_{t\in[0,1]}$ be a family of diffusions.
\begin{remark} \label{remark:expeqOU} \rm
For $y^n\in\cl C_{\alpha}([0,1])$ let $Z^{n,y^n}$ denote the solution
of equation (\ref{eq:OUy}) with $y$ replaced by $y^n$. Then if $y^n\to
y$ in $\cl C_{\alpha}([0,1])$,
from the generalized Wentzell--Friedlin theory
(Theorem 1 in \cite {Chi-Fis}), we have that
$((Z_{t}^{n,y^n})_{t\in[0,1]})_{n\in\N}$
satisfies the same LDP in $\cl{C}([0,1])$ as $((Z_{t}^{n,y})_{t\in
[0,1]})_{n\in\N}$.
\end{remark}

We now want to prove the lower semicontinuity of $J(\cdot|\cdot)$ on
$\cl{C}([0,1])\times\cl C_{\alpha}([0,1])$.

\begin{proposition} \label{proposition:lowsemiconOU} If
$(f^n,y^n)\rightarrow(f ,y )$ in $\cl{C}([0,1])\times\cl C_{\alpha}
([0,1])$, then
\[
\liminf_{n\to+\infty}J\bigl(f^n|y^n\bigr) \geq
J(f |y ).
\]
\end{proposition}
\begin{proof}
If $y^{n}\overset{\cl C_{\alpha}([0,1])}{\longrightarrow} y $,
then for any $\ep>0$, eventually $\inf_{t \in[0,1]} \lvert\frac
{y (t)}{y^n(t)} \rvert^2\geq(1-\ep)$, and by the lower
semicontinuity of $J(\cdot|y )$,
\begin{align*}
&\liminf_{(y^n, f^n) \rightarrow(y ,f )} J\bigl(f^n|y^n\bigr)\\
&\quad =
\liminf_{(y^n,
f^n) \rightarrow(y ,f )} \frac{1}{2} \int_{0}^{1}
\biggl\lvert \frac
{\dot{f}^{n}(t)-(a_{0}+a_{1}f^n(t))}{y^{n}(t)} \biggr\rvert^{2}dt
\\
&\quad=\liminf_{(y^n, f^n) \rightarrow(y ,f )} \frac{1}{2}\int_{0}^{1}
\biggl\lvert \frac{\dot{f}_{n}(t)-(a_{0}+a_{1}f^n(t))}{y
(t)} \biggr\rvert^{2}\cdot \biggl\lvert
\frac{y (t)}{y^n(t)} \biggr\rvert^2 dt
\\
&\quad\geq\liminf_{(y^n, f^n) \rightarrow(y ,f )} \frac{1}{2}\int_{0}^{1}
\biggl\lvert \frac{\dot{f}_{n}(t)-(a_{0}+a_{1}f^n(t))}{y
(t)} \biggr\rvert^{2} dt \cdot\inf
_{t \in[0,1]} \biggl\lvert\frac{y
(t)}{y^n(t)} \biggr\rvert^2
\\
&\quad=(1-\ep)\liminf_{ f^n \rightarrow f } J\bigl(f^n|y \bigr),
\end{align*}
and the proposition holds.
\end{proof}

\begin{theorem}
Consider the family of processes $(Y^n,Z^n)_{n \in\N}$, where
$(Y^n)_{n \in\N}$ is a family of processes with paths in $\cl
C_{\alpha}([0,1])$ and for $n\in\N$, $Z^n$ is the solution of (\ref
{eq:OU}). Suppose that $(Y^n)_{n \in\N}$ is independent from the
Brownian motion and satisfies a LDP with the good rate function $I_{Y}$
and the speed $\gamma(n)=n$. Then $(Y^n,Z^n)_{n \in\N}$ satisfies
the WLDP with the speed $\gamma(n)=n$ and rate function
\[
I(y,z)= I_Y(y) + J(z|y),
\]
and $(Z^n)_{n \in\N}$ satisfies the LDP with the speed $\gamma(n)$
and the rate function
\[
I_Z(z)= \inf_{y\in\cl C_{\alpha}([0,1]) }\bigl\{I_Y(y) +
J(z|y)\bigr\}.
\]
\end{theorem}

\begin{proof}
The family of processes $(Y^n,Z^n)_{n \in\N}$, thanks to
Remark \ref{remark:expeqOU} and Proposition \ref
{proposition:lowsemiconOU}, satisfies
the hypotheses of Theorem \ref{theorem:chaganty}, therefore the theorem
holds.
\end{proof}

Now let $(Z_{t}^{n,y})_{t\in[0,1]}$ be a family of continuous Gaussian
processes. We have to prove that $(Z_{t}^{n,y^n})_{t\in[0,1]}$
satisfies the same LDP as $(Z_{t}^{n,y})_{t\in[0,1]}$ when
$y^{n}\overset{\cl C_{\alpha}([0,1])}{\longrightarrow} y $.
Let $\tilde{Z}^{n,y^n}_t=Z^{n,y^n}_t-m(t)$
for every $n \in\N$ and $t \in[0,1]$.

Straightforward calculations show that there exists $L>0$, such that
\begin{align*}
&\sup_{s,t\in [ 0,1 ], s \neq t } n\cdot\frac{\lvert
k^{y^n} ( t,t ) + k^{y^n} ( s,s ) -2k^{y^n}
( t,s ) \rvert}{\lvert t-s \rvert^{2\alpha}}
\\
&\quad\leq L\underset{s,t\in [ 0,1 ], s \neq t } {\sup}
\frac{ ( e^{a_{1}(t-s)} - 1 )}{\lvert
t-s \rvert^{2\alpha}} < +\infty, \quad\mathrm{for}\ 2\alpha= 1.
\end{align*}
Therefore the family $ (\tilde{Z}^{n,y^n})_{n \in\N} $ is
exponentially tight at the speed $n$.
Furthermore, conditions (\ref{eq:meanlimit}) and (\ref{eq:covlimit})
of Theorem \ref{theorem:ldp-gaussian} are fullfilled, in fact
\[
\lim_{n \rightarrow+\infty} \mathbb{E} \bigl[ \bigl\langle\lambda ,
\tilde{Z}^{n,y^n}\bigr\rangle \bigr] = 0
\]
and
\[
\lim_{n\rightarrow+\infty} \Var \bigl( \bigl\langle\lambda, \tilde
{Z}^{n,y^n}\bigr\rangle \bigr)\cdot n = \int_{0}^{1}
\int_{0}^{1} k^y(s,t) \ d\lambda(t)\ d
\lambda(s),
\]
where
$k^y(s,t)= e^{a_{1}(s+t)} \int_{0}^{s\wedge t} e^{-2a_{1}u}y^2(u)\  du$.
Therefore $ (\tilde{Z}^{n,y^n})_{n \in\N} $ satisfies a LDP on $\cl
{C}([0,1])$.
Finally, thanks to the contraction principle, the family $
({Z}^{n,y^n})_{n \in\N} $ satisfies a LDP on $\cl{C}([0,1])$ with
the rate function $J(\cdot| y)$ defined in (\ref{eq:rate function_GP}).

\begin{remark} \rm The lower semicontinuity of $J(\cdot|\cdot)$ on
$\cl{C}([0,1])\times\cl C_{\alpha}([0,1])$ follows from Proposition
\ref{proposition:lowsemiconOU}.
\end{remark}

We have proved that the hypotheses of Theorem \ref{theorem:chaganty}
are verified, so the LDP for $({Z}^n)_{n \in\N} $ follows.

\section{Estimates of level crossing probability}\label{sect:ruin}
In this section we will study the probability of level crossing for a
family of conditionally Gaussian processes. In particular, we will study
the probability
\begin{equation}
p_n=\P\Bigl(\sup_{0\leq t \leq1}\bigl(Z^n_t-
\varphi(t)\bigr)>1\Bigr),
\end{equation}
as $ n \rightarrow\infty$, where $(Z^n)_{n \in\N}$ is a family of
conditionally Gaussian process. In this situation the probability $p_n$
has a large deviation limit
\[
\lim_{n \rightarrow\infty} \frac{1}{\gamma(n)}\log (p_n)=-I_{\varphi}.
\]
The main reference in this section is \cite{Bal-Pac}.
We now compute $\lim_{n \rightarrow\infty} \frac{1}{\gamma(n)}\log
(p_n)$, for a fixed continuous path $\varphi\in\cl{C}([0,1])$. The
computation is simple, in fact, since $(Z^n)_{n \in\N}$ satisfies a
LDP with
the rate function
\begin{equation}
I_Z(z)=\inf_{y \in C} \bigl\lbrace
I_Y(y)+J(z|y) \bigr\rbrace,
\end{equation}
where $I_Y(\cdot)$ is the rate function associated to the family of
conditioning processes $(Y^n)_{n \in\N}$, $C$ is the Polish set where
$(Y^n)_{n \in\N}$ takes values, and $J(\cdot|y)$ is the good rate
function of the family of Gaussian processes $(Z^{n,y})_{n \in\N}$.
If we denote
\[
\cl A=\Bigl\{ w \in\cl{C}\bigl([0,1]\bigr): \sup_{0\leq t\leq1}
\bigl(w(t)-\varphi (t)\bigr)>1\Bigr\},
\]
we have that
\begin{equation*}
\begin{aligned} -\inf_{w \in\mathring{\cl{A}}}I_Z(w)\leq
\liminf_{n \rightarrow
+\infty}\frac{1}{\gamma(n)}\log(p_n)\leq\limsup
_{n \rightarrow
+\infty}\frac{1}{\gamma(n)}\log(p_n)\leq-\inf
_{w \in\bar{\cl{A}}}I_Z(w) \end{aligned} %
\end{equation*}
where
\[
\bar{\cl{A}}= \Bigl\lbrace w \in\cl{C}\bigl([0,1]\bigr): \sup
_{0\leq t\leq
1}\bigl(w(t)-\varphi(t)\bigr) \geq1 \Bigr\rbrace\,
\]
and
\[
\mathring{\cl{A}}=\cl A= \Bigl\lbrace w \in\cl{C}\bigl([0,1]\bigr): \sup
_{0\leq t\leq1}\bigl(w(t)-\varphi(t)\bigr) > 1 \Bigr\rbrace.
\]
It is a simple calculation to show that
$\inf_{w \in\mathring{\cl{A}}}I_Z(w)=\inf_{w \in\bar{\cl{A}}}I_Z(w)$.
Therefore,
\begin{equation}
\label{eq: 1 Estimates } \lim_{n \rightarrow\infty}\frac{1}{\gamma(n)}\log(p_n)=-
\inf_{w
\in\cl{ A}}I_Z(w).
\end{equation}
For every $t \in[0,1]$ let
$\cl{A}_t= \lbrace w \in\cl{C}([0,1]): w(t)= 1+\varphi
(t) \rbrace$,
then
$\cl{A}=\bigcup_{0\leq t\leq1}\cl{A}_t$
and so
\[
\inf_{w \in\cl{A}}I_Z(w)=\inf_{y \in C}
\inf_{0\leq t\leq1}\inf_{w \in\cl{A}_t} \bigl\lbrace
I_Y(y)+J(w|y) \bigr\rbrace.
\]
\subsection{Gaussian process with random mean and variance}
For every $n \in\N$, let $Z^n=X^nY^n_1+ Y_2^n$ as in Section \ref
{sect:main}. In this case we know that
\begin{equation*}
J(z|y) =
\begin{cases}
\frac{1}{2} \lVert
\frac{z-y_2}{y_1} \rVert^{2}_{\bar{\cl
{H}}}, & \frac{z-y_2}{y_1}
\in\bar{\cl{H}},
\\
+\infty& \text{otherwise}.
\end{cases}
\end{equation*}
Therefore, we have
\begin{equation*}
\begin{aligned} \inf_{w \in\cl{A}}I_Z(w)&=
\inf_{y \in\cl C_{\alpha}([0,1])\times
\cl C([0,1])}\inf_{0\leq t\leq1}\inf_{w \in\cl{A}_t}
\bigl\lbrace I_Y(y)+J(w|y) \bigr\rbrace
\\
&=\inf_{y \in\cl C_{\alpha}([0,1])\times\cl C([0,1])}\inf_{0\leq
t\leq1}\inf
_{w \in\cl{A}_t} \biggl\lbrace I_Y(y)+\frac{1}{2}
\biggl\lVert\frac{w-y_2}{y_1} \biggr\rVert^{2}_{\bar{\cl{H}}} \biggr
\rbrace. \end{aligned} %
\end{equation*}
The set of paths of the form\querymark{Q2}
\[
h(u) =\int_{0}^{1} \bar{k}(u,v)\ d\lambda(v),
\quad u \in [0,1],\quad\lambda\in\cl{M}[0,1],
\]
is dense in $\bar{\cl{H}} $ and, therefore, the infimum $ \inf_{w \in\cl{A}_t} \lbrace I_Y(y)+\frac
{1}{2}\lVert\frac{w}{y}\rVert^{2}_{\bar{\cl{H}}} \rbrace$
is the same as that over the functions $w$ such that
\[
w(u)-y_2(u) =y_1(u)\cdot\int_{0}^{1}
\bar{k}(u,v)\ d\lambda(v), \quad u \in[0,1],
\]
for some $ \lambda\in\cl{M}[0,1]$. For such kind of paths,
recalling the expression of their norms in the RKHS, the functional we
aim to minimize is given by
\[
I_Y(y)+\frac{1}{2} \biggl\lVert\frac{w-y_2}{y_1} \biggr
\rVert^{2}_{\bar
{\cl{H}}}=I_Y(y)+\frac{1}{2}\int
_{0}^{1} \int_{0}^{1}
\bar{k}(u,v)\ d\lambda(u)\ d\lambda(v)
\]
therefore, it is enough to minimize the right-hand side of the above
equation with respect to the measure $\lambda$, with the additional
constraint that
\[
w(t)= 1+\varphi(t),
\]
which we can write in the equivalent form
\[
\int_{0}^{1} \bar{k}(t,v)\ d\lambda(v)-
\frac{1+\varphi
(t)-y_2(t)}{y_1(t)} =0.
\]
This is a constrained extremum problem, and thus we are led to use the
method of Lagrange multipliers. The measure $\lambda$ must be such that
\[
\int_{0}^{1} \int_{0}^{1}
\bar{k}(u,v)\ d\lambda(u)\ d\mu(v)=\beta \int_{0}^{1}
\bar{k}(t,v)\ d\mu(v),\quad \mu\in\cl{M}[0,1] ,
\]
for some $\beta\in\R$. We find
\[
\beta=\frac{\int_{0}^{1}\bar{k}(t,u)\ d\lambda(u)}{\bar
{k}(t,t)}=\frac{1+\varphi(t)-y_2(t)}{y_1(t)\bar{k}(t,t)}
\]
and
\[
\bar{\lambda}=\frac{1+\varphi(t)-y_2(t)}{y_1(t)\bar{k}(t,t)}\delta _{ \lbrace t \rbrace}.
\]
Such measure satisfies the Lagrange multipliers problem, and it is
therefore a critical point for the functional we want to minimize.
Since this is a strictly convex functional restricted on a linear
subspace of $ \cl{M}[0, 1]$, it is still strictly convex, and thus
the critical point $ \bar{\lambda} $ is actually its unique point of
minimum. Hence, we have
\begin{equation*}
\inf_{w \in\cl{A}}I_Z(w)=\inf_{y \in\cl C_{\alpha}([0,1])\times\cl
C([0,1])}
\inf_{0\leq t\leq1} \biggl\lbrace I_Y(y)+ \frac{
(1+\varphi(t)-y_2(t) )^2}{2y_1^2(t)\bar{k}(t,t)}
\biggr\rbrace.
\end{equation*}

\subsection{Ornstein--Uhlenbeck processes with random diffusion coefficient}
In this case
\begin{align*}
J(f|y) &=
\begin{cases}
\frac{1}{2}\lVert f-m\rVert^{2}_{\cl{H}_y}, & f-m\in
\cl{H}_y,\\
+\infty, & f-m\notin\cl{H}_y,
\end{cases}
\\
&=
\begin{cases}
\frac{1}{2}\int_{0}^{1}
\lvert \frac{\dot
{f}(t)-(a_{0}+a_{1}f(t))}{y(t)} \rvert^{2}dt, & f\in
H_{1}^{x},\\
+\infty, & f\notin H_{1}^{x},
\end{cases}
\end{align*}
where $m(t)=e^{a_{1}t}(x + \frac{a_0}{a_1}[ 1 - e^{-a_1 t}])$, $t \in
[0,1]$, $\cl{H}_y$ is the RKHS associated to
\[
k^y(s,t)= e^{a_{1}(s+t)} \int_{0}^{s\wedge t}
e^{-2a_{1}u}y^{2}(u)\ du,
\]
and $H_1^x =m+\cl{H}_y$.

We have
\begin{equation*}
\begin{aligned} \inf_{w \in\cl{A}}I_Z(w)&=
\inf_{y \in\cl C_{\alpha}[0,1]}\inf_{0\leq t\leq1}\inf_{w \in\cl{A}_t}
\bigl\lbrace I_Y(y)+J(w|y) \bigr\rbrace
\\
&=\inf_{y \in\cl C_{\alpha}[0,1]}\inf_{0\leq t\leq1}\inf
_{w \in\cl
{A}_t} \biggl\lbrace I_Y(y)+\frac{1}{2}
\lVert w-m\rVert^{2}_{\cl
{H}_y} \biggr\rbrace. \end{aligned}
\end{equation*}

The set of paths of the form\querymark{Q3}
\[
h(u) =\int_{0}^{1} k^y(u,v)\ d
\lambda(v), \quad u \in[0,1],\quad \lambda\in\cl{M}[0,1],
\]
is dense in $\cl{H}_y $, therefore, the infimum
\[
\inf_{w \in\cl{A}_t} \biggl\lbrace I_Y(y)+
\frac{1}{2}\lVert w-{m}\rVert^{2}_{\cl{H}_y} \biggr\rbrace,
\]
is the same as that over the functions of the form
\[
w(u) =m(u)+\int_{0}^{1} k^y(u,v)\ d
\lambda(v), \quad u \in[0,1],
\]
for some $ \lambda\in\cl{M}[0,1]$. For paths of such kind,
recalling the expression of their norms in the RKHS, the functional we
aim to minimize is given by
\[
I_Y(y)+\frac{1}{2}\lVert w-{m}\rVert^{2}_{\cl{H}_y}=I_Y(y)+
\frac
{1}{2}\int_{0}^{1} \int
_{0}^{1}k^y(u,v)\ d\lambda(u)\ d
\lambda(v)
\]
therefore, it is enough to minimize the right-hand side of the above
equation with respect to the measure $\lambda$, with the additional
constraint
\[
w(t)= 1+\varphi(t),
\]
which can be written in the equivalent form
\[
\int_{0}^{1} k^y(t,v)\ d
\lambda(v)+m(t)-\bigl(1+\varphi(t)\bigr) =0.
\]
This is a constrained extremum problem, and thus we are led to use the
method of Lagrange multipliers. We find
\[
\beta=\frac{\int_{0}^{1} k^y(t,u)\ d\lambda(u)}{ k^y(t,t)}=\frac
{1+\varphi(t)-m(t)}{ k^y(t,t)}
\]
and
\[
\bar{\lambda}=\frac{1+\varphi(t)-{m}(t)}{ k^y(t,t)}\delta_{
\lbrace t \rbrace},
\]
$ \delta_{ \lbrace t \rbrace} $ standing for the Dirac
mass in $ t $. Such measure satisfies the Lagrange multipliers problem,
and it is therefore a critical point for the functional we want to
minimize. Since this functional is a strictly convex restricted on a
linear subspace of $ \cl{M}[0, 1]$, it is still strictly convex, and
thus the critical point $ \bar{\lambda} $ is actually its unique
point of minimum. Hence, we have
\begin{equation*}
\inf_{w \in\cl{A}_t} I_Y(y)+\frac{1}{2}\lVert w-m
\rVert^{2}_{\cl{H}_y} =I_Y(y)+ \frac{ (1+\varphi(t)-m(t) )^2}{2 k^y(t,t)},
\end{equation*}
and therefore
\begin{equation*}
\inf_{w \in\cl{A}}I_Z(w)=\inf_{y \in\cl C_{\alpha}([0,1])}
\inf_{0\leq t\leq1} \biggl\lbrace I_Y(y)+ \frac{ (1+\varphi
(t)-m(t) )^2}{2 k^y(t,t)}
\biggr\rbrace.
\end{equation*}

%\begin{appendix}
%\end{appendix}

%\begin{acknowledgement}%[title={Acknowledgments}]
%\end{acknowledgement}

%\begin{funding}
%\gsponsor[id=,sponsor-id=]{}
%\gnumber[refid=]{}
%\end{funding}

%\begin{appendix}
%\end{appendix}

%\begin{acknowledgement}%[title={Acknowledgments}]
%\end{acknowledgement}

%\begin{funding}
%\gsponsor[id=,sponsor-id=]{}
%\gnumber[refid=]{}
%\end{funding}

%\begin{thebibliography}{}
%\end{thebibliography}

\end{document}